\documentclass[12pt,reqno]{amsart}
\usepackage{hyperref}
\usepackage{color, bm, amscd, tikz-cd}
\newcommand{\newsection}[1]{\setcounter{equation}{0} \section{#1}}
\newcommand{\bydef}{\stackrel{\rm def}{=}}

\newcommand{\cH}{{\mathcal H}}

\newcommand{\cK}{{\mathcal K}}

\newcommand{\cM}{{\mathcal M}}

\newcommand{\bD}{{\mathbb D}}

\newenvironment{psmallmatrix}
{\left(\begin{smallmatrix}}
	{\end{smallmatrix}\right)}

\newtheorem{thm}{Theorem}[section]

\newtheorem{lemma}[thm]{Lemma}

\newtheorem{proposition}[thm]{Proposition}

\theoremstyle{definition}

\newtheorem*{theorem}{Theorem}
\newtheorem{definition}[thm]{Definition}
\newtheorem{remark}[thm]{Remark}

\numberwithin{equation}{section}

\def\textmatrix#1&#2\\#3&#4\\{\bigl({#1 \atop #3}\ {#2 \atop #4}\bigr)}
\def\dispmatrix#1&#2\\#3&#4\\{\left({#1 \atop #3}\ {#2 \atop #4}\right)}

\author{Daniel Alpay, Tirthankar Bhattacharyya, Abhay Jindal and Poornendu Kumar}
\address{Department of Mathematics\\
	Chapman University\\
	Orange, CA 92866, USA.}
\email{alpay@chapman.edu}

\address{Department of Mathematics\\
	Indian Institute of Science\\
	Bangalore 560012, India}
\email{tirtha@iisc.ac.in; abjayj@iisc.ac.in; poornendukumar@gmail.com}

\usepackage{fancyhdr}
\begin{document}
	\title{A dilation theoretic approach to approximation by inner functions}
	\maketitle
	\begin{abstract}
		Using results from theory of operators on a Hilbert space, we prove approximation results for matrix-valued holomorphic functions on the unit disc and the unit bidisc.
		The essential tools are the theory of unitary dilation of a contraction and the realization formula for functions in the unit ball of $H^\infty$.  We first prove a generalization of a result of Carath\'eodory. This generalization has many applications. A uniform approximation result for matrix-valued holomorphic functions which extend continuously to the unit circle is proved using the Potapov factorization. This generalizes a theorem due to Fisher.
		Approximation results are proved for matrix-valued functions for whom a naturally associated kernel has finitely many negative squares. This uses the Krein-Langer factorization. Approximation results for $J$-contractive meromorphic functions where $J$ induces an indefinite metric on $\mathbb C^N$ are proved using the Potapov-Ginzburg Theorem. Moreover, approximation results for holomorphic functions on the unit disc with values in certain other domains of interest are also proved.
		
	\end{abstract}
	
	\section{Introduction}
	
	%Approximation of bounded holomorphic functions, defined on the unit disc or on the polydisc by rational inner functions or by their convex combinations have been extremely well-studied. We bring in the new viewpoint that these approximation results can be proven by operator theoretic techniques.
	
	Let $M_{N}(\mathbb{C})$ be the Banach algebra of all $N \times N$ complex matrices with the operator norm. For $\Omega = \mathbb D$ or $\Omega = \mathbb D^2$, a holomorphic function $F: \Omega \to M_{N}(\mathbb{C})$ is called $rational$ if every entry is a rational function with the poles off $\Omega$ and is called $inner$ if the boundary values of the function on the unit circle/torus are unitary matrices almost everywhere.
	
	Carath\'eodory, in his study of holomorphic functions from the open unit disc $\mathbb D = \{ z \in \mathbb C : |z| < 1 \}$ of the complex plane to the closed unit disc $\overline{\mathbb D}$, proved the following theorem, see Section 284 in \cite{C}. Later, Rudin generalized this to functions taking values in $\overline{\mathbb{D}}$ but defined on the polydisc $\mathbb D^n$, see \cite{Rudin}.
	\begin{theorem}[Carath\'eodory and Rudin] \label{CR}
		Let $\Omega$ denote the open unit disc $\mathbb D$ or the bidisc $\mathbb D^2$. Any holomorphic function $\varphi: \Omega \to \overline{\mathbb{D}}$ can be approximated (uniformly on compact subsets) by rational inner functions.
	\end{theorem}

	Such a function $\varphi$ is said to be in the {\em Schur class}. There is a proof of this theorem through the fact that any solvable Pick-Nevanlinna interpolation problem has a rational inner solution.  This technique carries over to matrix-valued functions.
	
	For decades now, the theory of bounded operators on Hilbert spaces has been successfully used to give new proofs of complex analytic theorems. Two prominent examples are Sarason's approach to $H^\infty$ interpolation \cite{Sarason} and Agler's proof of Lempert's theorem \cite{Agler-inv}. See also \cite{AMY}.
	
    We shall give a new proof of the theorem above in a more general setting, viz., when the target set $\overline{\mathbb D}$ is replaced by certain compact sets of interest in higher dimension. This includes matrix-valued functions. We shall use the {\em state space method}, a term coined in \cite{Kaashoek}, motivated by the huge contribution of linear system theory to function theoretic operator theory by the transfer function realization formula for operator-valued holomorphic functions on appropriate domains in $\mathbb{C}$ or $\mathbb{C}^{n}.$ See \cite{Bart} and \cite{Matrix-Poly}. The second tool in our kitty is a dilation theorem due to O. Nevanlinna \cite{NE}, greatly popularized later by Levy and Shalit in \cite{Levy-Shalit}.

    Carath\'eodory's (and Rudin's) theorem is striking because the approximants map the unit disc {\em onto} itself whereas the approximated function is only required to map the unit disc {\em into} itself. In our proof using the ``state space method'', the idea is to start with the fact that any Schur class function has a realization
    $$\varphi(z) = A + zB(I - zD)^{-1}C$$
    with the associated {\em system matrix (colligation)} $\textmatrix     	A & B \\ C & D \\$ contractive and then produce approximants $\varphi_{m}$ in terms of unitary colligations $\textmatrix A & B_{m} \\ C_{m} & D_{m} \\ $. We can ensure that these unitary colligations act on finite dimensional spaces, thereby making $\varphi_{m}$ rational and inner. The convergence question is converted into showing a matrix convergence: $$ B_{m} D_{m}^{k} C_{m}\rightarrow B D^{k} C \quad  \text{ as } \quad  m\rightarrow \infty, \quad \text { for all } k \geq 1.$$
	
	Carath\'eodory's theorem for matrix-valued functions and an appealing characterization of matrix-valued rational inner functions on the unit disc by Potapov lead us to a generalization of Fisher's theorem. Using the Blaschke product description of scalar rational inner functions, Fisher proved the following well-known result in \cite{Fisher}.
	\begin{theorem}[Fisher]
		Let $f$ be analytic on $\mathbb D$, continuous on $\overline{\mathbb D}$, and bounded by one. Then $f$ may be uniformly approximated on $\overline{\mathbb D}$ by convex combinations of finite Blaschke products.
	\end{theorem}
	
	Potapov showed in \cite{PO} that any $N \times N$ matrix-valued rational inner function $\Phi$ is of the form
	\begin{align*}
		\Phi (z) = U \prod\limits_{m=1}^{M} \Big(b_{\alpha_m} (z) P_{m} + (I_{\mathbb{C}^{N}} - P_{m})\Big) \text{ for } z \in \mathbb D,
	\end{align*}
	where $M$ is a natural number, $U$ is an $N \times N$ unitary matrix, the $P_m$ are projections onto certain subspaces of $\mathbb{C}^{N}$, the $\alpha_m$ are points in the open unit disc  and
	$$b_{\alpha}(z):= \frac{z-\alpha}{1-\overline{\alpha} z} \text{ for } \alpha\in\mathbb{D}$$
	stands for a Blaschke factor. Such functions came to be known as Blaschke-Potapov products with a function of the form $b_{\alpha} P_{\cM} + (I_{\mathbb{C}^{N}} - P_{\cM})$ being called a Blaschke-Potapov factor because  $b_{\alpha}$ is a Blaschke factor.
	
	As one of the principal applications of Theorem \ref{Mat-Den}, we shall reap a {\em uniform} approximation result for matrix-valued holomorphic functions on $\mathbb D$ which are continuous on $\overline{\mathbb D}$ as well. This generalizes  Fisher's theorem. The crucial input which makes this possible is the Blaschke-Potapov formula. This is done in Section 3.
	
	Theorem \ref{Mat-Den} has further applications. The fact that a matrix-valued contractive holomorphic function $F$ satisfies $I \ge F(z)F(z)^*$ as well as, equivalently,
	$$K_F(z,w)=\frac{I - F(z)F(w)^*}{1 - z\overline{w}} \succeq 0,$$
	where $K \succeq 0$ for a kernel means that it is positive semi-definite, leads to generalizations in two different directions. Relaxing the positivity condition, we prove the following in Section 4.
	The proof of this uses the Krein-Langer Theorem.
	
	One way to study non self adjoint operators is through their characteristic functions. This inexorably leads to $J$-contractive functions, where $J\in\mathbb C^{N\times N}$ is a signature matrix, i.e., $J=J^{-1}=J^*$, see page 62 of \cite{MR48:904} for example. We have approximation results for $J$-contractive meromorphic functions as well as for functions whose kernel corresponding to $J$ (analogous to $K_F$ above, but now $J$ replacing the identity operator) has finitely many negative squares. The terminologies are explained in the relevant section.
	
	We also have two results about functions taking values into the symmetrized bidisc $\Gamma$ or into the tetrablock $\overline{\mathbb{E}}$. The sets $\Gamma$ and $\overline{\mathbb E}$ as well as the $\Gamma$-inner functions and the $\overline{\mathbb{E}}$-inner functions will be described in the context in the final section when we prove the results.
	
	We thank the referees for valuable comments which have greatly improved the paper.
	\section{Approximation by dilation} \label{appbydil}
	
	We start with a proposition which is a slight improvement of Lemma 6.2 of \cite{GN}. It has the same proof and the proof also carries verbatim in the case the function $F$ takes its values in rectangular matrices instead of square ones.
	
	\begin{proposition} \label{poly approx on polydisc}
		Any holomorphic function $F: \mathbb{D}^{n} \to M_{N}(\mathbb{C})$ with $\|F(z)\| < 1$ for all $z\in\mathbb{D}^{n},$ can be approximated (uniformly on compact subsets) by matrix-valued polynomials ${P_m}$ with $\|P_m\|_{\infty, \overline{\mathbb{D}^{n}}} < 1,$ for all $m\geq{1}$.
	\end{proposition}

	%\begin{proposition}
	%Given a contractive square matrix $\begin{bmatrix}
	%A & B \\ C & D
	%\end{bmatrix}$ on  $ \mathbb{C}^{N} \oplus \mathbb{C}^{d}$ and a diagonal matrix $Z=\begin{bmatrix}
	%z_1I_{d_1} & 0 \\
	%0 & z_2I_{d_2}
	%\end{bmatrix}$ with $d_1+d_2=d,$ there exist finite dimensional Hilbert spaces $\cK_{m},$ unitary matrices $\begin{bmatrix}
	%A & B_{m} \\ C_{m} & D_{m}
	%\end{bmatrix} $ on $ \mathbb{C}^{N} \oplus \cK_m,$ and diagonal matrices $Z_{m} = \begin{bmatrix}
	%Z & 0 \\ 0 & *
	%\end{bmatrix}$ on $\cK_{m}$ such that
	%$$B_{m} Z_m(D_{m}Z_m)^{k} C_{m} \to B Z(DZ)^{k} C$$ (in norm) for all $k\geq 1.$
	%\end{proposition}
	
	We now quote a useful tool.

	\begin{thm}[Realization formula for the disc]\label{matrix valued 1}
		Let $F:\mathbb{D} \to M_{N}(\mathbb{C})$ be a rational function such that $\|F\|_{\infty} \leq 1$. Then there exist a positive integer $d$ and a contractive matrix
		$$\begin{bmatrix}
			A & B \\ C & D
		\end{bmatrix} : \mathbb{C}^{N} \oplus \mathbb{C}^{d} \to \mathbb{C}^{N} \oplus \mathbb{C}^{d} $$ such that
		$$ F(z) = A + zB (I - z D)^{-1}C.$$
	\end{thm}
	
	This finite-dimensional realization formula is the disc version of the celebrated Kalman-Yakubovich-Popov lemma; see \cite{DDGK} for an indefinite
        version of it and \cite{GN}, Proposition 4.2 for a recent proof. These two proofs give different points of view. See also \cite{Ball}.

        % very well-known and can be proved in many ways.
        %This folklore result can be found in chapter 4 of \cite{Ball} or
	
	The next result is the most crucial step towards proving the main theorem.

	\begin{thm}\label{rational approx}
		Any rational function $F: \mathbb{D} \to M_{N}(\mathbb{C})$ with $\|F(z)\| \leq 1$ for all $z \in \mathbb D$ can be approximated (uniformly on compact subsets) by $M_{N}(\mathbb{C})$-valued rational inner functions.
	\end{thm}
	\begin{proof}
		The sequence of $M_{N}(\mathbb{C})$-valued rational inner functions which approximates $F$ will actually be constructed by mixing two ingredients.
		First we invoke the Realization Formula, viz., Theorem \ref{matrix valued 1} and set some notations. Let $T$ denote the contraction
		$$ \begin{bmatrix}
			A & B \\ C & D
		\end{bmatrix} : \mathbb{C}^{N} \oplus \mathbb{C}^{d} \to \mathbb{C}^{N} \oplus \mathbb{C}^{d}$$
		with $D_{T^{*}}$ and $D_{T}$ being the defect operators $(I- T T^{*})^{1/2}$ and $(I - T^{*} T)^{1/2}$ respectively. Let
		$$ D_{T^{*}} = \begin{bmatrix}
			S_{1} & S_{2} \\ S_{3} & S_{4}
		\end{bmatrix} \text{ and } D_{T} = \begin{bmatrix}
			T_{1} & T_{2} \\ T_{3} & T_{4}
		\end{bmatrix}$$
		as operators from $\mathbb{C}^{N} \oplus \mathbb{C}^{d}$ into itself. Let $\mathcal{H} := \mathbb{C}^{N} \oplus \mathbb{C}^{d}.$ The second ingredient is a finite dilation of the contraction $T,$ i.e., for any $m\geq1,$ a space $\cH_{m}$ consisting of the direct sum of $(m+1)$ copies of $\cH$ and a unitary $U_{m}$ on it such that $T^{j} = P_{\cH} U_{m}^{j}|_{\cH}$ for $j=1,\dots,m.$ This idea originated with \cite{NE}, see also \cite{Levy-Shalit}. A sequence of functions $F_{m}$ induced by the unitaries $U_{m}$ will be the approximating sequence.

		To that end, consider the space
		$$\cK_m \bydef \mathbb{C}^{d} \oplus \mathcal{H}\oplus \dots \oplus \mathcal{H}\oplus \mathbb{C}^{N} \oplus \mathbb{C}^{d},$$
		where $\mathcal{H} $ occurs $(m-1)$ times. Now, consider the block operator matrix
		$$ U_{m} :=  \begin{bmatrix}
			A & B & 0 &\dots & 0 & S_{1} & S_{2} \\
			C & D & 0 &\dots & 0 & S_{3} & S_{4} \\
			T_{1} & T_{2} & 0 & \dots & 0 & -A^{*} & -C^{*}\\
			T_{3} & T_{4} & 0 & \dots & 0 & -B^{*} & -D^{*}\\
			0 & 0 & I_{\mathcal{H}} & \dots & 0 & 0 & 0 \\
			\vdots & \vdots & \vdots & \ddots & \vdots & \vdots & \vdots \\
			0 & 0 & 0 & \dots & I_{\mathcal{H}} & 0 & 0
		\end{bmatrix} $$
		acting on the space $ \mathbb{C}^{N} \oplus \cK_m$. A straightforward calculation will show that $U_{m} $ is a unitary matrix. This $U_m$ is our $ \begin{bmatrix}
			A & B_{m} \\ C_{m} & D_{m}
		\end{bmatrix}$ alluded to in the introduction. We note that
		$$B_{m}= \begin{bmatrix} B & 0 &  \dots & 0 & S_1 & S_2
		\end{bmatrix}, \;\; C_{m} = \begin{bmatrix}
			C & T_{1} & T_{3} & 0 & \dots & 0
		\end{bmatrix}^{t},$$ and
		$$ D_{m} = \begin{bmatrix}
			D & 0 &\dots & 0 & S_{3} & S_{4} \\
			T_{2} & 0 & \dots & 0 & -A^{*} & -C^{*}\\
			T_{4} & 0 & \dots & 0 & -B^{*} & -D^{*}\\
			0 & I_{\mathcal{H}} & \dots & 0 & 0 & 0 \\
			\vdots & \vdots & \ddots & \vdots & \vdots & \vdots \\
			0 & 0 & \dots & I_{\mathcal{H}} & 0 & 0
		\end{bmatrix}.$$
		For fixed $k\geq 1$, we shall show that
\begin{equation} \label{MatConv} B_{m} D_{m}^{k}C_{m} = B D^{k} C \end{equation} for all $m \geq k+2$. First note that for $m \geq 3$, the matrix $D_{m} C_{m}  : \mathbb{C} ^{N} \to \mathbb{C}^{d} \oplus \mathbb{C}^{N} \oplus \mathbb{C}^{d} \oplus \mathbb{C}^{N} \oplus \mathbb{C}^{d} \oplus\dots \oplus \mathbb {C}^N\oplus\mathbb{C}^d$  is given by
		$$ \begin{bmatrix}
			DC & T_{2}C & T_{4} C & T_{1} & T_{3} & 0 & \dots & 0
		\end{bmatrix}^{t}.$$
		Also, a simple calculation gives the following
		$$ D_m^kC_m=\begin{bmatrix}
			DC & * & * & * & \dots & *& 0  &0
		\end{bmatrix}^{t},  \quad \text{for  }  m\geq k+2,$$
		where the asterisk symbols mean that certain matrices are there which do not enter later computation. A matrix multiplication then yields \eqref{MatConv}.

	To summarize, we have proved that there is a sequence of finite-dimensional Hilbert spaces $\mathcal H_m$, viz., the direct sum of $m+1$ copies of $\mathcal H$ and a sequence of unitary matrices $U_m$ on them satisfying a convergence property as follows.
		$$ \mathcal H_m = \mathbb{C}^{N} \oplus \cK_{m} \text{ and } \begin{bmatrix}
			A & B_{m} \\ C_{m} & D_{m}
		\end{bmatrix} : \mathbb{C}^{N} \oplus \cK_m \to \mathbb{C}^{N} \oplus \cK_m $$
		and $B_{m} D_{m}^{k} C_{m} \to B D^{k} C$ (in norm) for all $k\geq 1$. We are ready to define the approximants.
		
		Consider the matrix-valued functions $F_m$ defined as
		$$F_m(z)= A+zB_m(I-zD_m)^{-1}C_m.$$
		The functions $F_m$ are rational inner because $ \begin{bmatrix}
			A & B_{m} \\ C_{m} & D_{m}
		\end{bmatrix}$ are unitary matrices. Fix a compact set $S\subset\mathbb{D}$. For given $\epsilon>0$, there exists $M_{0}\in\mathbb{N}$ such that
		$$|z|^l<\epsilon \quad \quad\text{for all}\quad l\geq M_{0} \text{ and }z\in S.$$
		Now,
		\begin{align*}
			&	\|F(z)-F_m(z)\| \\
			= & \|z\sum_{k\geq 0}^\infty(B_mD_m^kC_m-BD^kCz^k)\|\\
			\leq & |z| \sum_{k\geq 0}^\infty\|B_mD_m^kC_m-BD^kC\| |z|^k \\
			= & |z| \sum_{k\geq 0}^{M_{0} -1}\|B_mD_m^kC_m-BD^kC\| |z|^k + |z| \sum_{k\geq M_{0}}^\infty\|B_mD_m^kC_m-BD^kC\| |z|^k \\
			\leq & |z| \sum_{k\geq 0}^{M_{0} -1}\|B_mD_m^kC_m-BD^kC\| |z|^k +  \epsilon \frac{2 |z|}{1- |z|}\\
			= & \epsilon \frac{2 |z|}{1- |z|} \hspace{5mm} (\text{ for all } m \geq M_{0} +1).
		\end{align*}
		Therefore, the sequence of rational inner functions $F_{m}$ converges uniformly on compact subsets of $\mathbb{D}.$	
	\end{proof}

	\begin{thm}\label{Mat-Den1}
		Any holomorphic function $F: \mathbb{D} \to M_{N}(\mathbb{C})$ with $\|F(z)\| \leq 1$ for all $z\in\mathbb{D}$ can be approximated (uniformly on compact subsets) by $M_{N}(\mathbb{C})$-valued rational inner functions.
	\end{thm}
	\begin{proof}
		By maximum norm principle, [Theorem 2, \cite{Condori}], either $\|F(z)\| <1$ for all $z\in\mathbb{D},$ or $\|F(z)\| \equiv 1.$
		
		\textbf{Case-1:} $\|F(z)\| <1$ for all $z\in\mathbb{D}.$\\
		In this case, Proposition \ref{poly approx on polydisc} and Theorem \ref{rational approx} together will give us an approximation of $F$ by matrix-valued rational inner functions.
		
		\textbf{Case-2:} $\|F(z)\| \equiv 1.$ \\
		By Theorem 4 of \cite{Condori}, there are $N \times N$ constant unitary matrices $U$ and $V,$ and an analytic function $G: \mathbb{D} \to M_{N-1}$ with $\|G(z)\| \leq 1$ for all $z\in\mathbb{D},$ such that
		\begin{align} \label{structure}
			F (z) = U \begin{bmatrix}
				1 & 0 \\ 0 & G(z)
			\end{bmatrix} V.
		\end{align}
		So, if $N=2,$ then Caratheodory's Theorem together with the equation \eqref{structure} will give us an approximation of $F$ by matrix-valued rational inner functions. Inductively, we can prove the result for $N>2.$
		\end{proof}

The approximation theorem above continues to hold for matrix-valued functions on the bidisc. We shall outline the proof below. The finite dimensional realization formula we need has recently been proven by Knese in \cite{GN}.

	\begin{thm}[Realization formula for the bidisc] \label{matrix valued 2}
		Let $F:\mathbb{D}^2 \to M_{N}(\mathbb{C})$ be a rational function such that $\|F\|_{\infty} \leq 1$. Then there exist positive integers $d_1 ,d_2$ and a contractive matrix
		$$\begin{bmatrix}
			A & B \\ C & D
		\end{bmatrix} : \mathbb{C}^{N} \oplus \mathbb{C}^{d} \to \mathbb{C}^{N} \oplus \mathbb{C}^{d} \hspace{5mm} \text{ with } d = d_1+d_2$$
		such that with the notation $Z= z_1I_{d_1} \oplus  z_2I_{d_2}$, we have
		$$ F(z_1, z_2) = A + BZ (I -  DZ)^{-1}C.$$
		
	\end{thm}
	
	\begin{thm}\label{Mat-Den}
		Any holomorphic function $F: \bD^{2} \to M_{N}(\mathbb{C})$ with $\|F(z_{1},z_{2})\| \leq 1$ for all $(z_{1}, z_{2})\in\mathbb{D}^{2}$ can be approximated (uniformly on compact subsets) by $M_{N}(\mathbb{C})$-valued rational inner functions.
	\end{thm}
	
	\begin{proof}
	Let $F:\mathbb{D}^2\rightarrow M_N(\mathbb{C})$ be a holomorphic map with $\|F(z_1, z_2)\|\leq{1}$ for all $(z_1, z_2)\in\mathbb{D}^2$. Then in view of Lemma 6.1 of \cite{GN}, it is enough to consider the case when $\|F(z_1, z_2)\| <1$ for all $(z_1, z_2)\in\mathbb{D}^2.$ Now by Proposition \ref{poly approx on polydisc}, we can take $F$ to be a polynomial. Invoke Theorem \ref{matrix valued 2} to get positive integers $d_1 ,d_2$ and a contractive matrix
	$$ T = \begin{bmatrix}
		A & B \\ C & D
	\end{bmatrix} : \mathbb{C}^{N} \oplus \mathbb{C}^{d} \to \mathbb{C}^{N} \oplus \mathbb{C}^{d} \hspace{5mm} \text{ with } d = d_1+d_2$$
	such that with the notation $Z= z_1I_{d_1} \oplus  z_2I_{d_2}$, we have
	$$ F(z_1, z_2) = A + BZ (I -  DZ)^{-1}C.$$
	Consider the $m$-unitary dilation $\begin{psmallmatrix}
	A & B_m \\ C_m & D_m
	\end{psmallmatrix}$ of $T$ on $\mathbb{C}^N\oplus\cK_{m}$. A matrix multiplications then yields that
	$$B_mZ_m(D_mZ_m)^kC_m=BZ(DZ)^kC , \quad\text{ for } m\geq k+2 $$
	 where $Z_m=\operatorname{diag}\left(Z, *, *, \dots, *\right)$ be any diagonal operator acting on $\cK_{m}$ and the asterisk symbols stand for diagonal matrices whose diagonal entries are either $z_1$ or $z_2$ or $e^{i\theta}$ for some $\theta$. 	Consider the matrix-valued rational inner functions  $F_m$ on $\bD^2$ defined as
	$$F_m(z_1, z_2)= A+B_mZ_m(I-D_mZ_m)^{-1}C_m.$$ A similar argument as in the case of the disc will give that the sequence of rational inner functions $F_{m}$ converges to $F$ uniformly on compact subsets of $\mathbb{D}^2.$

	 \end{proof}
	
	\begin{remark}
	A comment about the case of the polydisc $\bD^{n}$ is in order for $n>2.$ Let $\varphi$ be a function from the Schur-Agler class, i.e., $\varphi$ is in $H^\infty_n$ in the notation of \cite{Agler}. Let $\{z_{1}, z_{2},\dots\}$ be a countable dense subset of $\bD^{n}$. Consider for every $m \geq 1,$ the solvable Pick-Nevanlinna interpolation data $\{(z_{1}, \varphi(z_{1})), \dots, (z_{m}, \varphi(z_{m}))\}$. It is known that this has a rational inner solution $\varphi_{m}$ from the Schur-Agler class. Montel's theorem then proves that there is a subsequence of $\{\varphi_{m}\}$ converging to $\varphi$ uniformly over compact subsets of $\bD^{n}$. This technique carries over to matrix-valued functions of Schur-Agler class. This matrix-valued version of Rudin's result is not known if $\varphi$ is in Schur class because the Schur class is bigger than the Schur-Agler class. Also, the state space method cannot be applied to prove the result even for the Schur-Agler class because a finite realization for rational inner functions on the polydisc is not known. This is a limitation for the state space method.
	\end{remark}

	\newsection{The convex hull of matrix-valued rational inner functions on the disc}
	From now on, our functions will be on $\mathbb D$. It follows from Potapov's work that every matrix-valued rational inner function is holomorphic in a neighbourhood of the closed unit disc $\overline{\bD}.$ In this section, we shall give a description of the closed convex hull of the matrix-valued rational inner functions generalizing the theorem in \cite{Fisher}.
	
	Let $F$ be an $M_{N}(\mathbb{C})$-valued function which is holomorphic in $\bD$ and continuous on $\overline{\bD}$. For $0\leq r \leq 1,$ set
	\begin{equation} \label{rscale} F_{r}(z) := F(rz) \hspace{5mm} (z\in\overline{\bD}). \end{equation}
	Clearly, $F_{r}$ is holomorphic in $\bD$ and continuous on $\overline{\bD}$ for any $0\leq r\leq 1.$ The following two lemmas follow from direct calculations.
	
	\begin{lemma}\label{mult}
		Let $\Phi,\Psi$ be two $M_{N}(\mathbb{C})$-valued rational inner functions. Suppose for some fixed $r\in [0,1],$ $\Phi_{r},\Psi_{r}$ can be written as convex combination of rational inner functions, then $(\Phi \Psi)_{r}$ can also be written as convex combination of rational inner functions.
	\end{lemma}
	
	\begin{lemma}\label{unitary}
		Let $\Phi$ be an $M_{N}(\mathbb{C})$-valued rational inner functions and $U\in M_{N}(\mathbb{C})$ be a unitary. Suppose for some fixed $r\in [0,1],$ $\Phi_{r}$ can be written as convex combination of rational inner functions, then $U \Phi_{r}$ can also be written as convex combination of rational inner functions.
	\end{lemma}
	
	\begin{lemma} \label{phi_r}
		If $\Phi$ is any $M_{N}(\mathbb{C})$-valued rational inner function, then for any $0 \leq r \leq 1,$ $\Phi_{r}$ can be written as convex combination of $M_{N}(\mathbb{C})$-valued rational inner functions.
	\end{lemma}
	\begin{proof}
		Note that if $\varphi$ is a scalar-valued rational inner function and $P$ is an orthogonal projection of $\mathbb{C}^{N}$ onto some subspace, then the matrix-valued function $\varphi P + (I_{\mathbb{C}^{N}} - P)$ is also rational inner. For a Blaschke factor $b$ and for any $0 \leq r \leq 1,$ it follows from \cite{Fisher} that $b_{r}$, as defined in \eqref{rscale} can be written as a convex combination of scalar-valued rational inner functions. So the $M_{N}(\mathbb{C})$-valued holomorphic function $b_{r} P + (I_{\mathbb{C}^{N}} -  P)$ can be written as a convex combination of $M_{N}(\mathbb{C})$-valued rational inner functions. The rest of the proof follows from Lemma \ref{mult} and Lemma \ref{unitary}.
	\end{proof}
	
	\begin{lemma}\label{F_r}
		Let $F$ be an $M_{N}(\mathbb{C})$-valued function which is holomorphic in $\bD$ and continuous on $\overline{\bD}$. Then $F_{r}$ converges uniformly to $F$ on $\overline{\bD}$ as $r \rightarrow 1.$
	\end{lemma}
	\begin{proof}
		If $F$ is scalar-valued, then it follows from Mergelyan's theorem. Since $M_{N}(\mathbb{C})$ is finite dimensional, all norms on $M_{N}(\mathbb{C})$ are equivalent. So there exists a positive constant $c_{N}$ such that
		$$\|A\| \leq c_{N} \max\limits_{i,j} |a_{ij}|$$
		for all $A = [a_{ij}]_{N \times N},$ where $\|A\|$ is the operator norm of matrix $A.$ Let $F = [F_{ij}]_{N \times N}.$ Let $\epsilon >0$ be given. Since each $F_{ij}$ is scalar-valued, there exists $r$ close to $1$ such that
		$$| F_{ij}(z) - F_{ij}(rz) | < \frac{\epsilon}{c_{N}}$$
		for all $z\in\overline{\bD}$ and for all $i,j.$ So we get
		$$\| F(z) - F(rz)\| \leq c_{N} \max\limits_{i,j} | F_{ij}(z) - F_{ij}(rz) | < \epsilon$$
		for all $z \in \overline{\bD}.$ This concludes the proof.
	\end{proof}
	
	We are now ready with the generalization of Fisher's theorem.
	
	\begin{thm} \label{GenFish}
		Let $F$ be an $M_{N}(\mathbb{C})$-valued function which is holomorphic in $\bD$ and continuous on $\overline{\bD}$. Suppose $\|F(z)\| \leq 1$ for all $z\in\mathbb{D}.$ Then $F$ can be uniformly approximated on $\overline{\bD}$ by convex combinations of $M_{N}(\mathbb{C})$-valued rational inner functions.
	\end{thm}

	\begin{proof}
		Let $\epsilon >0$ be given. By Lemma \ref{F_r}, there exists $r\in(0,1)$ such that
		$$\|F - F_{r}\|_{\infty, \overline{\bD}} < \frac{\epsilon}{2}.$$
		Let $\mathbb{D}_{r}$ be the closed unit ball of radius $r$ centered at $0.$ By Theorem \ref{Mat-Den}, there exists an $M_{N}(\mathbb{C})$-valued rational inner function $\Phi$ such that
		$$\| F - \Phi\|_{\infty, \mathbb{D}_{r}} < \frac{\epsilon}{2}.$$
		This implies
		$$\| F_{r} - \Phi_{r}\|_{\infty, \overline{\bD}} < \frac{\epsilon}{2}.$$
		So we get
		$$\| F - \Phi_{r}\|_{\infty, \overline{\bD}} < \epsilon.$$
		By Lemma \ref{phi_r}, it follows that $\Phi_{r}$ itself is a convex combination of $M_{N}(\mathbb{C})$-valued rational inner functions.
	\end{proof}
	
	\newsection{Relaxing analyticity}
	\subsection{Meromorphic functions}
	\begin{thm} \label{mero_intro}
		Let $F$ be an $M_{N}(\mathbb{C})$-valued meromorphic function on $\mathbb D$. Suppose the kernel $K_F(z,w) =\frac{I - F(z)F(w)^*}{1 - z\overline{w}}$ has finitely many negative squares. Let $A(F) \subset \mathbb D$ be the set on which $F$ is analytic. Then $F$ can be approximated uniformly on compact subsets of $A(F)$ by rational functions which are unitary matrix-valued on the unit circle. Moreover if $F$ is continuous on the unit circle, then $F$ can be approximated uniformly on the unit circle $\mathbb{T}$ by convex combinations of quotient of matrix-valued rational inner functions.
	\end{thm}
	
	\begin{proof}There is a remarkable factorization of operator-valued functions, due to Krein and Langer, for those functions $F$ which satisfy that $K_F$ has finitely many negative squares, see \cite{KL}, \cite{DLS}. Since our function is matrix-valued, the Krein-Langer factorization in this context says that there exists a Blaschke-Potapov product $B$ of degree $k$ and a matrix-valued holomorphic function on the disc $L$ such that
		$$F(z) = B(z)^{-1} L(z)$$
		and  $\|L(z)\| \leq 1$ for all $z \in {\mathbb{D}}.$ We apply Theorem \ref{Mat-Den} to get a sequence $\{L_m\}$ of matrix-valued rational inner functions converging to $L$  uniformly on compact subsets of $\mathbb D$. Then, the sequence $B(z)^{-1} L_m(z)$ does the job.
		
		In the case when $F$ is continuous on the unit circle, we use holomorphicity of $B$  in a neighbourhood of $\overline{\mathbb{D}}$ to conclude that the $L$ obtained above is continuous on $\mathbb{T}$ and $\|L(z)\| \leq 1$ for all $z\in\mathbb{T}.$  Now we invoke Theorem \ref{GenFish} to get  a sequence of convex combinations of matrix-valued rational inner functions $\{L_{m}\}$ such that $L_{m}$ converges to $L$ uniformly on $\overline{\mathbb{D}}.$ Define
		$$F_{m}(z) = B(z)^{-1}L_{m}(z).$$ Consider
		$$
		\| F(z) - F_{m}(z) \| = \| B(z)^{-1} (L(z) - L_{m}(z)) \|  \leq \| B(z)^{-1}\| \| L(z) - L_{m}(z) \| .
		$$
		Since $B$ is continuous on $\mathbb{T},$ $F_{m}$ converges to $F$ uniformly on $\mathbb{T}.$
		
		If we apply the right Krein-Langer factorization, then $$F(z) = R(z) \tilde{B}(z)^{-1}.$$ By a similar calculation $R_{m}(z) \tilde{B}(z)^{-1}$ will approximate $F$ uniformly on $\mathbb{T}$ where $R_{m}$ approximates $R$ as in Theorem \ref{GenFish}. That completes the proof of Theorem \ref{mero_intro}. \end{proof}
	
	\subsection{$J$-contractive functions}
	In a new direction of generalization, we consider the case of indefinite metric in the coefficient space $\mathbb C^N$, that is kernels of the form
	\[
	\frac{J-F(z)JF(w)^*}{1-z\overline{w}}
	\]
	where $J\in\mathbb C^{N\times N}$ is a signature matrix. Such a matrix is unitarily equivalent to $J_0$ defined by
	\[
	J_0=\begin{bmatrix}I_p&0\\0&-I_q\end{bmatrix},\quad p+q=N,
	\]
	with $J_0=I_N$ if $q=0$ and $J_0=-I_N$ if $p=0$.
	We are interested in the case $p>0$, $q>0$.
	In the sequel we focus on the case $J=J_0$. The formulas presented are valid for arbitrary $J$ (for which
	$p>0$ and $q>0$ in the corresponding $J_0$).

	We will use the Potapov-Ginzburg transform (see \cite{adrs,MR48:904}), which allows to reduce to the case $J=J_0=I_N$.
	Following \cite{MR1638044}, we set
	\[
	P=\frac{I_N+J_0}{2}\quad{\rm and}\quad Q=\frac{I_N-J_0}{2}.
	\]
	For $J=J_0$ at hand, we have
	\[
	P=\begin{bmatrix}I_p&0\\0&0\end{bmatrix}\quad{\rm and}\quad Q=\begin{bmatrix}0&0\\0&I_q\end{bmatrix}
	\]
	and
	\[
	P+QF(z)=\begin{bmatrix}I_p&0\\
		F_{21}(z)&F_{22}(z)\end{bmatrix}.
	\]
	Writing $F=\begin{bmatrix}F_{11}&F_{12}\\ F_{21}&F_{22}\end{bmatrix}$ we will assume that $\det F_{22}\not\equiv 0$.
	\begin{definition}
		The Potapov-Ginzburg transform of $F$ is given by
		\begin{equation*}
			\Sigma(z)=\left(PF(z)+Q\right)(P+QF(z))^{-1},
		\end{equation*}
		at those points where the inverse exists, with inverse given by
		\begin{equation*}
			F(z)=(P-\Sigma(z)Q)^{-1}(\Sigma(z)P-Q).
		\end{equation*}
	\end{definition}
	
	The following formulas hold. See \cite[p. 66]{MR1638044}, \cite{MR2002b:47144}.
	
	\begin{eqnarray}
		\Sigma(z)&=&(P-F(z) Q)^{-1}(F(z) P-Q) \nonumber \\
		F(z)&=&(Q+P\Sigma(z))(P+Q\Sigma(z))^{-1} \nonumber \\
		\label{neg-k}
		I_N-\Sigma(z)\Sigma(w)^*    &=&(P-F(z)Q)^{-1}\left(J_0-F(z)J_0F(w)^*\right)(P-F(w)Q)^{-*}
		\label{P-G-kern-1}
		\\
		I_N-\Sigma(w)^*\Sigma(z)    &=&(P+Q F(w))^{-*}\left(J_0-F(w)^*J_0F(z)\right)(P+QF(z))^{-1} \nonumber
	\end{eqnarray}
	
	A function $F$ meromorphic in $\mathbb D$ is called {\em $J_0$-contractive} if
	\begin{equation*}
		\label{J0inner}
		F(z)J_0F(z)^*\le J_0
	\end{equation*}
	at each point of analyticity of $F$ in $\mathbb D$. Such a function is in particular of bounded type in $\mathbb D$ and admits non-tangential limits almost everywhere on the unit circle. A matrix $A$ is called {\em $J_0$-unitary} if $AJ_0A^* = J_0$. A rational function $F$ will be called {\em $J_0$-inner} if the limiting values exist and are $J_0$-unitary everywhere on the unit circle except possibly at a finite number of points. In the following theorem, we mention a special case first for the sake of better exposition. \smallskip

	%$J_0$-contractive functions
	%in the open unit disc, and in particular $J_0$-inner rational functions
	%may have poles inside the open unit disc
	%or on the unit circle; these correspond to
	%Blaschke-Potapov factors of the second and third kind respectively. The latter are also called Brune factors; see \cite{MR85m:94012}.\smallskip

	\begin{thm} \label{J-Intro} \leavevmode
		\begin{enumerate}
			\item Let $F$ be $J_0$-contractive, with domain of analyticity $A(F)\subset\mathbb D$. Then,
			%$\det F_{22}\not=0$, and
			$F$ can be approximated uniformly on compact subsets by rational $J_0$-inner functions.
			\item Let $F$ be meromorphic in the open unit disc with domain of analyticity $A(F)\subset \mathbb D$ such that the kernel
			\[
			\frac{J_0-F(z)J_0F(w)^*}{1-z\overline{w}}
			\]
			has a finite number of negative squares in $A(F)$. Then $F$ can be approximated uniformly on compact subsets of $A(F)$
			by rational $J_0$-inner functions.
		\end{enumerate}
	\end{thm}
	
	\begin{proof} The Potapov-Ginzburg transform of $F$ exists by \cite[Theorem 1.1, p. 14]{Dym_CBMS}. By \eqref{P-G-kern-1}, $\Sigma$ is contractive and meromorphic in the open unit disc,
		and hence contractive and analytic there (the contractivity implies that the isolated singularities of $F$ are removable). Applying Theorem
		\ref{Mat-Den} to $\Sigma$ we can write $\Sigma=\lim_{m\rightarrow\infty}B_m$, where the $B_m$ are finite Blaschke products and where the
		convergence is uniform on compact subsets of the open unit disc. Writing $B_m=((B_m)_{ij})_{i,j=1}^2$ where $(B_m)_{22}$ is
		$\mathbb  C^{q\times q}$-valued, we have in particular
		\[
		\lim_{m\rightarrow\infty}\det (B_m)_{22}=\det \Sigma_{22}
		\]
		and in particular  $\det B_m\not\equiv 0$ for $m$ large enough. It follows that the inverse Potapov-Ginzburg transforms, say $F_m$,
		of the $B_m$ exist for such $m$. The functions $F_m$ are rational and $J_0$-inner. That completes the proof of part $(1)$.

		We now consider the case of negative squares and recall that its Potapov-Ginzburg transform, say $\Sigma$,
		is well defined (see e.g. \cite[Theorem 6.8]{ad3}). It follows from \eqref{neg-k} that
		\[
		\frac{    I_N-\Sigma(z)\Sigma(w)^*}{1-z\overline{w}}    =(P-F(z)Q)^{-1}\frac{J_0-F(z)J_0F(w)^*}{1-z\overline{w}}(P-F(w)Q)^{-*}
		\]
		and in particular the kernel $\frac{    I_N-\Sigma(z)\Sigma(w)^*}{1-z\overline{w}}$ has a finite number of negative squares in the open unit disc.
		We apply Theorem \ref{mero_intro} to $\Sigma$, and we
		get an approximation for $F$ by taking the inverse Potapov-Ginzburg transform. Thus, we have proved part $(2)$ of Theorem \ref{J-Intro}. \end{proof}
	
	\newsection{$\Gamma$-valued and $\overline{\mathbb E}$-valued functions}
	Let $\Omega$ be a bounded polynomially convex domain. The {\em distinguished boundary} $b\Omega$ is the smallest closed subset of $\overline{\Omega}$ on which every continuous function on $\overline{\Omega}$ that is analytic in $\Omega$ attains its maximum modulus.
	
	\begin{definition}
		A rational $\overline{\Omega}-$ inner function is a rational analytic map $x:\mathbb{D}\rightarrow\overline{\Omega}$ with the property that $x$ maps $\mathbb{T}$ into the distinguished boundary $b\Omega$ of $\Omega$. The degree, deg($x$), of a rational $\overline{\Omega}-$ inner function is defined to be the maximum of degree of each components.
	\end{definition}
	
	This section deals with functions which take values into the symmetrized bidisc
	$$\Gamma = \{(z+w, zw): |z| \le 1, |w| \le 1\}$$
	or into the tetrablock
	$$\overline{\mathbb{E}}=\{(a_{11}, a_{22}, \det(A)): A=\begin{bmatrix}
		a_{11} & a_{12} \\
		a_{21} & a_{22}
	\end{bmatrix} \text{ satisfies } \|A\| \le 1\}.$$
	The sets $\Gamma$ and $\overline{\mathbb{E}}$ are  non-convex and polynomially convex domains. The symmetrized bidisc was introduced by  Agler and Young in \cite{AY}  and the tetrablock  was introduced by Abouhajar, White and Young in \cite{AWY}. A great deal of function theory and operator theory has been done on these two domains. The following criteria will be useful. Let $\mathbb G$ be the open symmetrized bidisc.
	
	\begin{proposition}\cite{AY}
		Let $(s, p)\in\mathbb{C}^2$.
		The point $(s, p)\in\mathbb{G}$ (respectively $\Gamma$) if and only if
		$$|s|< (\text{respectively } \le ) 2, \text{ and } |s-\overline{s}p|< (\text{respectively } \le ) 1-|p|^2.$$
		The point $(s, p)\in b\mathbb{G}$ if and only if $|s|\leq{2}, |p|=1, \text{ and } s=\overline{s}p.$
	\end{proposition}
	
	There are similar criteria about the tetrablock.
	
	\begin{proposition}\cite{AWY}
		Let $(x_1, x_2, x_3)\in\mathbb{C}^3$.
		The point $(x_1, x_2, x_3)\in\mathbb{E}$ (respectively $\overline{\mathbb{E}}$) if and only if $$ |x_1-\overline{x_2}x_3|+|x_2-\overline{x_1}x_3|< (\text{respectively } \leq)  1-|x_3|^2.$$
		The point $(x_1, x_2, x_3)\in b\mathbb{E}$ if and only if $ x_1=\overline{x_2}x_3|, |x_3|=1, \text{ and } |x_2|\leq{1}.$
	\end{proposition}
	
	Algebraic and geometric aspects of rational $\Gamma-$ inner functions were studied in \cite{ALY-Adv}.
	For details about rational $\Gamma-$ inner functions and rational $\overline{\mathbb{E}}-$ functions, see \cite{ALY-JMAA-Symm, ALY-Adv, Als-Lyk-Inter,  Omar-Lykova}.

	\begin{proposition}%\label{Symm-Den}
		Any holomorphic function $h=(s,p): \mathbb{D} \to \Gamma$ can be approximated (uniformly on compact subsets) by rational $\Gamma$-inner functions.
	\end{proposition}
	
	\begin{proof}
		
		Let $h =(s,p):\mathbb{D}\to\Gamma$ be a holomorphic function. Invoke Proposition 6.1 of \cite{ALY-SIAM} to obtain an analytic function $F:\mathbb{D}\rightarrow M_{2}(\mathbb{C})$ with $\|F(\lambda\|\leq 1$ for all $\lambda\in\mathbb{D}$ such that
		\begin{align*}
			h=(\operatorname{tr} F, \det F).
		\end{align*}
		By Theorem \ref{Mat-Den}, there exists a sequence of matrix-valued rational inner functions $\{F_m\}$ on $\mathbb{D}$ which approximates $F$ uniformly on compact subsets of $\mathbb{D}$. For each $m\in\mathbb{N}$, consider the holomorphic functions $h_m:\mathbb{D}\to \Gamma$ defined as
		\begin{align*}
			h_{m}:= (\operatorname{tr} F_m, \det F_m).
		\end{align*}
		It is easy to see that $h_m$ are rational functions.
		
		To prove that $h_m$ are $\Gamma$-inner functions, we only need to make the elementary observation that for a unitary matrix $A$, the eigenvalues $\lambda_1$ and $\lambda_2$ lie in $\mathbb{T}$. So, $(\operatorname{tr} A, \det A)=(\lambda_1+\lambda_2, \lambda_1\lambda_2)\in b\Gamma$. Since $F_m$ are inner, $F_{m}(\lambda)$ are unitaries a.e. on the circle. Thus, $h_m$ are $\Gamma$-inner functions.
		
		Since $F_{m}$ converges to $F$ uniformly on compact subsets, it follows that $(F_{m})_{ij}$ converges to $F_{ij}$ uniformly on compact subsets. Therefore, $h_{m}$  converges to $h$ uniformly on compact subsets of $\mathbb{D}$.
	\end{proof}
	
	We remark that the method of proof of Carath\'eodory's theorem through Pick-Nevanlinna interpolation can also be applied to approximate holomorphic functions from $\mathbb D$ into the symmetrized bidisc because of a result of Costara, see Theorem 4.2 in \cite{Costara}.
	
	\begin{proposition}%\label{Tetra-Den}
		Any holomorphic function $x= (x_{1},x_{2},x_{3}): \mathbb{D} \to \overline{\mathbb{E}}$ can be approximated (uniformly on compact subsets) by rational $\overline{\mathbb{E}}$-inner functions.
	\end{proposition}
	
	\begin{proof}
		
		Let $x= (x_{1},x_{2},x_{3})$ be as in the above Theorem. By Lemma 7 of \cite{EKZ}, there exists an analytic function $F:\mathbb{D}\rightarrow M_{2}(\mathbb{C})$ with $\|F(\lambda\|\leq 1$ for all $\lambda\in\mathbb{D}$ such that
		\begin{align*}
			x=(F_{11}, F_{22}, \det F)
		\end{align*}
		where $F=[F_{ij}]_{i,j=1}^2$. Again by Theorem \ref{Mat-Den}, there exists a sequence of matrix-valued rational inner functions $\{F_m\}$ on $\mathbb{D}$ which approximates $F$ uniformly on compact subsets of $\mathbb{D}$. For $m\in\mathbb{N}$, define the holomorphic maps $x_{m}:\mathbb{D}\to\overline{\mathbb{E}}$ by
		$$x_m=((F_m)_{11}, (F_m)_{22}, \det F_m).$$
		Now we shall prove that  this $x_m$ will do our job. It is easy to see that $x_m$ are rational functions. Now we shall prove that $x_m$ are $\overline{\mathbb{E}}$-inner functions. Since $F_m$ are inner, $F_{m}(\lambda)$ are unitaries a.e. on the circle. It follows that $x_m(\lambda)\in b\mathbb{E}$ a.e. $\lambda\in\mathbb{T}$, see Theorem 7.1 of \cite{AWY}. Thus, $x_m$ are rational $\overline{\mathbb{E}}$-inner functions.
		
		Since $F_m$ converges uniformly on compact subsets to $F$, $(F_m)_{11}, (F_m)_{22},$ and $\det F_m$ converges uniformly on compact subsets to
		$F_{11}, F_{22}$ and $\det F$ respectively. Hence $x_m$ converges uniformly on compact subsets of $\mathbb{D}$ to $x$. This completes the proof.
	\end{proof}
	
	Acknowledgement:
	D. Alpay thanks the Foster G. and Mary McGaw Professorship in Mathematical Sciences, which supported this research.
	T. Bhattacharyya is supported by a J C Bose Fellowship JCB/2021/000041 of SERB and A. Jindal is supported by the Prime Minister's Research Fellowship PM/MHRD-20-15227.03. This research is supported by the DST FIST program-2021 [TPN-700661].

\end{document}